\documentclass[10pt]{article}

\usepackage{amssymb,amsthm,amsmath}
\usepackage{enumerate}
\usepackage{graphicx,color}

\newcommand{\dd}{\mathrm{d}}
\newcommand{\E}{\mathbb{E}}

\newcommand{\R}{\mathbb{R}}

\newcommand{\p}[1]{\mathbb{P}\left( #1 \right)}
\newcommand{\scal}[2]{\left\langle #1, #2 \right\rangle}

\DeclareMathOperator{\vol}{vol}

\usepackage[paper=a4paper, left=1.4in, right=1.4in, top=1.2in, bottom=1.5in]{geometry}
\newtheorem{theorem}{Theorem}
\newtheorem{lemma}[theorem]{Lemma}

\theoremstyle{remark}

\theoremstyle{definition}

\begin{document}

\begin{center}
\begin{large}
ON A CONVEXITY PROPERTY OF SECTIONS OF THE CROSS-POLYTOPE \\
\end{large}
\vspace{1em}
Piotr Nayar and Tomasz Tkocz
\end{center}

\bigskip

\begin{footnotesize}
\begin{center}
\parbox{0.75\textwidth}{
\noindent\textbf{Abstract.}
We establish the log-concavity of the volume of central sections of dilations of the cross-polytope (the strong B-inequality for the cross-polytope and Lebesgue measure restricted to an arbitrary subspace).
}
\end{center}

\bigskip

\noindent {\em 2010 Mathematics Subject Classification.} Primary 52A40; Secondary 60E15

\noindent {\em Key words.} cross-polytope, volume of sections, logarithmic Brunn-Minkowski inequality
\end{footnotesize}

\bigskip

\section{Introduction}

The conjectured logarithmic Brunn-Minkowski inequality posed by B\"or\"oczky, Lutwak, Yang and Zhang in \cite{BLYZ} can be equivalently stated as the following property of sections of the  cube $B_1^n = [-1,1]^n$: for every subspace $H$ of $\R^n$ the function
\[
(t_1,\ldots,t_n) \mapsto \vol_H\left( \text{diag}(e^{t_1},\ldots,e^{t_n})B_\infty^n  \cap H \right)
\]
is log-concave on $\R^n$ (see the papers by Saroglou \cite{Sar1} and \cite{Sar2} for a similar and other reformulations). Here $\text{diag}(e^{t_1},\ldots,e^{t_n})$ denotes as usual the $n\times n$ diagonal matrix with $e^{t_i}$ on the diagonal and $\vol_H$ denotes Lebesgue measure on $H$. In this note, we show that such a property holds for sections of the cross-polytope $B_1^n = \{x \in \R^n, \ \sum_{i=1}^n |x_i| \leq 1\}$.

\begin{theorem}\label{thm:Bineql1}
Let $H$ be a subspace of $\R^n$. Then the function
\[
(t_1,\ldots,t_n) \mapsto \vol_H\left(\mathrm{diag}(e^{t_1},\ldots,e^{t_n})B_1^n \cap H \right)
\]
is log-concave on $\R^n$.
\end{theorem}

In other words, the so-called strong B-inequality holds for $B_1^n$ and the (singular) measure being Lebesgue measure restricted to an arbitrary subspace of $\R^n$ (see the pioneering work \cite{CFM} and see \cite{Sar2} for connections to the logarithmic Brunn-Minkowski inequality). We shall present in the sequel a simple example of a symmetric log-concave measure for which the strong B-property fails. Cordero-Erausquin and Rotem have recently analysed in detail the strong B-property for centred Gaussian measures providing further examples of such measures (see \cite{CR}).

It can be checked directly (and will also be clear from our proof) that the same holds true when $B_1^n$ is replaced with $B_2^n$. We conjecture that the above theorem in fact holds for any $B_p^n$ ball put in place of $B_1^n$, $p > 1$.

\section{Proofs}

\subsection{Auxiliary results}

The heart of our argument is the following probabilistic formula for volume of sections of the cross-polytope.

\begin{lemma}\label{lm:volformula}
Let $H$ be a codimension $k$ subspace of $\R^n$. Let $u_1,\ldots,u_k$ be an orthonormal basis of the orthogonal complement of $H$ and let $v_1,\ldots,v_n$ be the column vectors of the $k\times n$ matrix formed by taking $u_1,\ldots, u_k$ as its rows. Then for any positive numbers $a_1,\ldots,a_n$ we have
\[
\vol_{H}\left( \mathrm{diag}(a_1,\ldots,a_n)B_1^n \cap H \right) = \frac{2^{n-k}}{(n-k)!\cdot \pi^{k/2}}\left(\prod_{j=1}^n a_j\right)\E\left[\frac{1}{\sqrt{\det\left(\sum_{j=1}^n a_j^{2}Y_j v_jv_j^T\right)}}\right],
\]
where $Y_1,\ldots,Y_n$ are i.i.d. standard one sided exponential random variables.
\end{lemma}
\begin{proof}
The starting point is a well-known integral representation for volumes of sections: for an even, homogeneous and continuous function $N\colon \R^n \to [0,\infty)$ vanishing only at the origin and $p > 0$ we have
\[
\Gamma(1+(n-k)/p)\vol_{n-k}(\{x \in \R^n, \ N(x) \leq 1\} \cap H) = \lim_{\varepsilon \to 0} \frac{1}{\varepsilon^k} \int_{H(\epsilon)} e^{-N(x)^p} \dd\!\vol_n(x),
\]
where $H$ is, as in the assumptions of the lemma, a codimension $k$ subspace of $\R^n$ whose orthogonal complement has an orthonormal basis $u_1,\ldots,u_k$ and 
\[
H(\epsilon) = \{x \in \R^n, \ |\scal{x}{u_j}| \leq \varepsilon/2, \ j = 1,\ldots,k\}.
\]
This fact was probably first used in \cite{MP} and in this generality appeared for instance in \cite{Bar} (Lemma 21). Its proof is based on Fubini's and Lebesgue's dominated convergence theorems. Using it for $p = 1$ and $N(x) = \sum a_i^{-1}|x_i|$, we get
\[
(n-k)!\cdot\textrm{vol}_{H}\left( \mathrm{diag}(a_1,\ldots,a_n)B_1^n \cap H \right) = \lim_{\varepsilon \to 0} \frac{1}{\varepsilon^k} \int_{H(\varepsilon)} e^{-\sum a_i^{-1}|x_i|} \dd x.
\]
Let $X_1,\ldots,X_n$ be i.i.d. standard two-sided exponential random variables, that is with density $\frac{1}{2}e^{-|x|}$. Then the vector $(a_1X_1,\ldots,a_nX_n)$ has the density $\frac{1}{2^n\prod a_i}\exp(-\sum a_i^{-1}|x_i|)$, so
\begin{align*}
&(n-k)!\cdot\textrm{vol}_{H}\left( \mathrm{diag}(a_1,\ldots,a_n)B_1^n \cap H \right) \\
&\qquad= 2^n\left(\prod a_i\right)\lim_{\varepsilon \to 0} \varepsilon^{-k} \p{ (a_1X_1,\ldots,a_nX_n) \in H(\varepsilon) } \\
&\qquad= 2^n\left(\prod a_i\right) \lim_{\varepsilon \to 0} \varepsilon^{-k} \p{ |\textstyle\sum_{i=1}^n a_iX_iu_{j,i}| \leq \varepsilon/2, \ j = 1,\ldots,k}.
\end{align*}
Let us compute the probability above and then the limit. Recall the classical fact that the $X_i$ are Gaussian mixtures (see also \cite{ENT}). More preciesly, each has the same distribution as the product $R_i\cdot G_i$ where the $G_i$ are standard Gaussian random variables and $R_i$ are i.i.d. positive random variables distributed as $\sqrt{2Y_i}$ with $Y_i$ being i.i.d. standard one-sided exponentials (see a remark following Lemma 23 in \cite{ENT}). If we condition on the $R_i$ and introduce vectors $\tilde u_j = [a_iR_iu_{j,i}]_{i=1}^n$ we thus get
\[
\p{ |\textstyle\sum_{i=1}^n a_iX_iu_{j,i}| \leq \varepsilon/2, \ j = 1,\ldots,k} = \p{ |\scal{G}{\tilde u_j}| \leq \varepsilon/2, \ j = 1,\ldots,k },
\]
where $G = (G_1,\ldots,G_n)$ is a standard Gaussian random vector. Let $V$ be the subspace spanned by $\tilde u_1, \ldots, \tilde u_k$ and $P_V$ the projection onto $V$.  Then $G_V = P_VG$ is a standard Gaussian random vector on $V$. The above probability thus equals $\p{G_V \in \varepsilon K}$, where $K$ is the subset of $V$ given by $K = \{x \in \R^n \cap V, \ |\scal{x}{\tilde u_j}| \leq 1/2, \ j = 1,\ldots,k\}$, therefore it equals
\[
\p{ |\scal{G}{\tilde u_j}| \leq \varepsilon/2, \ j = 1,\ldots,k } = \p{G_V \in \varepsilon K} = \varepsilon^k(2\pi)^{-k/2}\vol_k(K) + o(\varepsilon^k).
\]
We plug this back, use Lebesgue's dominated convergence theorem (notice that the function $\varepsilon^{-k}\p{G_V \in \varepsilon K}$ is majorised by $(2\pi)^{-k/2}\vol_k(K)$) and obtain
\begin{align*}
(n-k)!\cdot\vol_{n-k}\left(\{x \in \R^n, \ \sum a_i|x_i| \leq 1\} \cap H \right) &= 2^n\left(\prod a_i\right)\lim_{\varepsilon \to 0} \varepsilon^{-k} \E_R\p{ G_V \in \varepsilon K} \\
&= 2^n(2\pi)^{-k/2}\left(\prod a_i\right)\E_R\vol_k(K).
\end{align*}
We are almost done. It remains to recall an elementary fact that an intersection of exactly $n$ strips in $\R^n$, say $\bigcap_{j=1}^n\{x \in \R^n, \ |\scal{x}{v_j}| \leq 1/2\}$ is an image of the cube $[-1/2,1/2]^n$ under the linear map $(V^T)^{-1}$, where $V$ is the matrix whose columns are the $v_j$ (that is $V$ maps the $e_j$ onto $v_j$). Therefore the $n$-volume of the intersection is $\frac{1}{\det(V)}$. In other words, the volume is the reciprocal of the volume of the paralleletope $\{\sum t_iv_i, \ t_1,\ldots,t_n \in [0,1]\}$. Thus,
\[
\vol_{k}(K) = \frac{1}{\sqrt{\det(\tilde U^T \tilde U)}},
\]
where $\tilde U$ is the $n\times k$ matrix whose columns are the $\tilde u_j$. Noticing that the rows of $\tilde U$ are the vectors $a_{i}R_iv_i$ finishes the proof, since then
\[
\frac{1}{\sqrt{\det(\tilde U^T \tilde U)}} = \frac{1}{\sqrt{\det(\sum a_iR_i^2v_iv_i^T})}
\]
and as mentioned earlier $R_i$ has the same distribution as $\sqrt{2Y_i}$.
\end{proof}

We need the following standard lemma, whose proof can be found for example in \cite{Bap} (see Lemma 1 and Lemma 2 (vi) therein).

\begin{lemma}\label{lm:detchaos}
Let $A_1,\ldots,A_n$ be $k \times k$ real symmetric positive semidefinite matrices. Then the function
\[
(x_1,\ldots,x_n) \mapsto \det \left(\sum_{i=1}^n x_i A_i  \right)
\]
is of the form
\[
\sum_{1 \leq j_1, \ldots, j_k \leq n} b_{j_1,\ldots,j_k} x_{j_1}\cdot\ldots\cdot x_{j_k},
\]
where $b_{j_1,\ldots,j_k} =D(A_{j_1}, \ldots,A_{j_k})$ is the mixed discriminant of $A_{j_1}, \ldots,A_{j_k}$. In particular, $b_{j_1,\ldots,j_k} \geq 0$.
\end{lemma}

\begin{lemma}\label{lm:detlogconvex}
Let $v_1,\ldots,v_n$ be vectors in $\R^k$. Then the function
\[
(t_1,\ldots,t_n) \mapsto \log\det (\sum e^{t_i}v_iv_i^T)
\]
is convex on $\R^n$.
\end{lemma}
\begin{proof}
By Lemma \ref{lm:detchaos}, the function $f(t_1,\ldots,t_n) = \det (\sum e^{t_i}v_iv_i^T)$ is of the form
\[
f(t_1,\ldots,t_n) = \sum_{1 \leq j_1,\ldots, j_k \leq n} b_{j_1\ldots,j_k}e^{t_{j_1}+\ldots+t_{j_k}},
\]
for some nonnegative $b_{j_1\ldots,j_k}$. By H\"older's inequality,
\[
f(\lambda s + (1-\lambda)t) \leq f(s)^\lambda f(t)^{1-\lambda},
\]
which finishes the proof.
\end{proof}

\subsection{Proof of Theorem \ref{thm:Bineql1}}

Thanks to Lemma \ref{lm:volformula}, it suffices to show that the function
\[
\E\left[\det(\sum e^{t_i}Y_i v_iv_i^T)\right]^{-1/2} = \int_{(0,\infty)^n} \left[\det(\sum e^{t_i}y_i v_iv_i^T)\right]^{-1/2} e^{-\sum y_i} \dd y.
\]
is log-concave.
We do the same change of variables $y_i = e^{s_i}$ as in \cite{ENT} in the proof of the B-inequality for the exponential measure (Theorem 14). This gives
\[
\int_{\R^n} \left[\det(\sum e^{t_i+s_i} v_iv_i^T)\right]^{-1/2} e^{-\sum (e^{s_i}-s_i)} \dd y.
\] 
By Lemma \ref{lm:detlogconvex} the integrand is a log-concave function of $(s,t)$ on $\R^{2n}$ and by virtue of the Pr\'ekopa-Leindler inequality its marginal is also log-concave. \hfill$\square$

\section{Strong B-property}

We say that a Borel measure $\mu$ on $\R^n$ satisfies the strong B-inequality if for every symmetric convex set $K$ in $\R^n$ the function
\[
(t_1,\ldots,t_n) \mapsto \mu(\text{diag}(e^{t_1},\ldots,e^{t_n})K)
\]
is log-concave on $\R^n$. Nontrivial examples of such measures include standard Gaussian measure and the product symmetric exponential measure (see \cite{CFM} and \cite{ENT}). We remark that it is not true that every symmetric log-concave measure satisfies the strong B-inequality (see also \cite{CR}). Take a uniform measure $\mu$ on the parallelogram 
\[
K = \text{conv}\{(-1,-2),(-1,-1),(1,1),(1,2)\}\]
in $\R^2$. Let $K_t = \text{diag}(1,e^t)K$ and consider the function $f(t) = \log \mu(K_t) = \log\frac{| K_t \cap K|}{|K|}$. 
Clearly, $\max f = f(0) = 0$. Moreover, $\lim_{t\to -\infty} f(t) = -\infty$ (since $K_t \cap K$ converges to the interval $[-\frac{1}{3},\frac{1}{3}]\times \{0\}$) and $\lim_{t\to \infty} f(t) > -\infty$ (since $K_t \cap K$ converges to the parallelogram $\text{conv}\{(-\frac{1}{3},-\frac{2}{3}),(-\frac{1}{3},0),(\frac{1}{3},0),(\frac{1}{3},\frac{2}{3})\}$). Such a function cannot be concave.

\begin{center}
\begingroup%
  \makeatletter%
  \providecommand\color[2][]{%
    \errmessage{(Inkscape) Color is used for the text in Inkscape, but the package 'color.sty' is not loaded}%
    \renewcommand\color[2][]{}%
  }%
  \providecommand\transparent[1]{%
    \errmessage{(Inkscape) Transparency is used (non-zero) for the text in Inkscape, but the package 'transparent.sty' is not loaded}%
    \renewcommand\transparent[1]{}%
  }%
  \providecommand\rotatebox[2]{#2}%
  \newcommand*\fsize{\dimexpr\f@size pt\relax}%
  \newcommand*\lineheight[1]{\fontsize{\fsize}{#1\fsize}\selectfont}%
  \ifx\svgwidth\undefined%
    \setlength{\unitlength}{375bp}%
    \ifx\svgscale\undefined%
      \relax%
    \else%
      \setlength{\unitlength}{\unitlength * \real{\svgscale}}%
    \fi%
  \else%
    \setlength{\unitlength}{\svgwidth}%
  \fi%
  \global\let\svgwidth\undefined%
  \global\let\svgscale\undefined%
  \makeatother%
  \begin{picture}(1,0.5)%
    \lineheight{1}%
    \setlength\tabcolsep{0pt}%
    \put(0,0){\includegraphics[width=\unitlength,page=1]{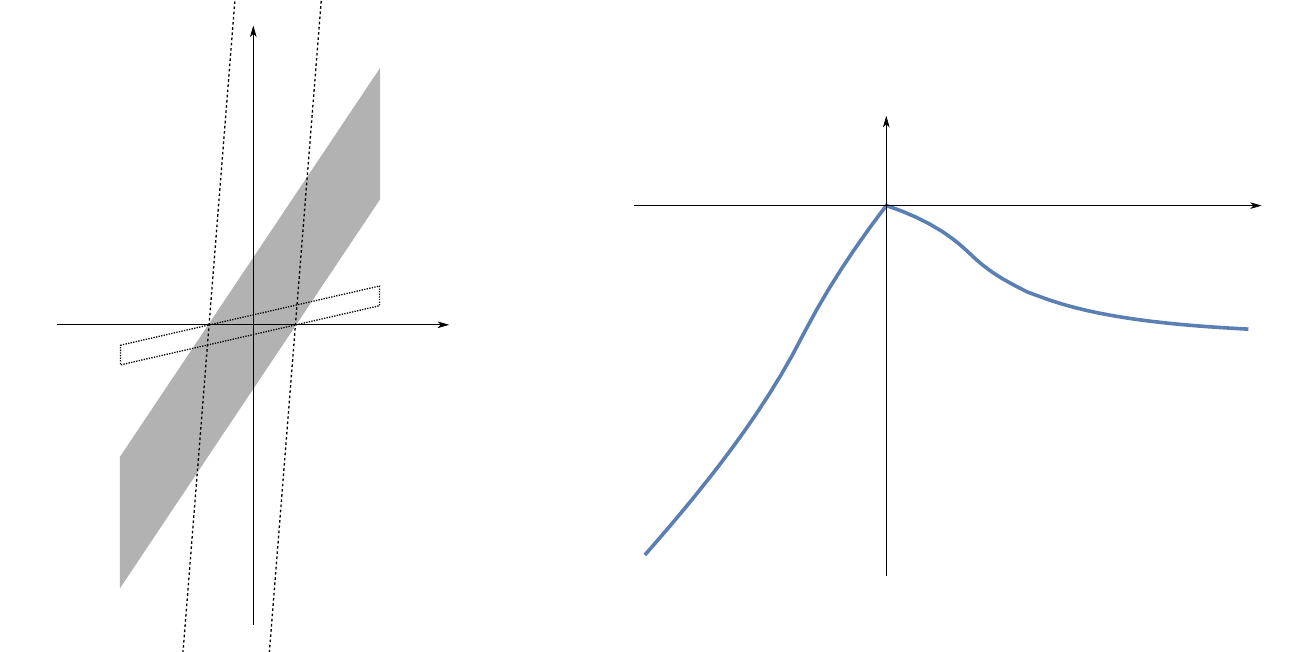}}%
    \put(-0.1358112,0.39649582){\color[rgb]{0,0,0}\makebox(0,0)[lt]{\begin{minipage}{0.22149075\unitlength}\raggedright \end{minipage}}}%
    \put(0.38282355,0.355479){\color[rgb]{0,0,0}\makebox(0,0)[lt]{\begin{minipage}{0.13034229\unitlength}\raggedright \end{minipage}}}%
    \put(0.24920876,0.47839418){\color[rgb]{0,0,0}\makebox(0,0)[lt]{\begin{minipage}{0.06734251\unitlength}\raggedright $K_2$\end{minipage}}}%
    \put(0.30112169,0.40385362){\color[rgb]{0,0,0}\makebox(0,0)[lt]{\begin{minipage}{0.15170771\unitlength}\raggedright $K=K_0$\end{minipage}}}%
    \put(0.29656427,0.2853606){\color[rgb]{0,0,0}\makebox(0,0)[lt]{\begin{minipage}{0.11200525\unitlength}\raggedright $K_{-2}$\end{minipage}}}%
    \put(0.86624217,0.23614044){\color[rgb]{0,0,0}\makebox(0,0)[lt]{\begin{minipage}{0.06734251\unitlength}\raggedright $f(t)$\end{minipage}}}%
  \end{picture}%
\endgroup%

\end{center}

\section{Another formula for volume of sections}

Using the same probabilistic representation of the double-sided exponential distribution, we shall derive a complementary formula to the one from Lemma \ref{lm:volformula}.

\begin{lemma}\label{lm:volformula-2}
Let $H$ be a $k$-dimensional subspace of $\R^n$ spanned by vectors $u_1,\ldots,u_k$ in $\R^n$ and let $v_1,\ldots,v_n$ be the column vectors of the $k\times n$ matrix formed by taking $u_1,\ldots, u_k$ as its rows. Then for any positive numbers $a_1,\ldots,a_n$ we have
\begin{align*}
\vol_{H}& \left( \mathrm{diag}(a_1,\ldots,a_n) B_1^n \cap H \right)  \\
&= \frac{2^{k}}{k!\cdot \pi^{(n-k)/2}} \sqrt{\det\left(\sum_{i=1}^n v_iv_i^T\right)} \E \left[\frac{1}{\sqrt{\prod_{i=1}^n Y_i}} \frac{1}{\sqrt{\det\left(\sum_{i=1}^n \frac{1}{Y_ia_i^2}v_iv_i^T\right)}} \right],
\end{align*}
where $Y_1,\ldots,Y_n$ are i.i.d. standard one sided exponential random variables.
\end{lemma}
\begin{proof}
Let 
\[
K = \left\{y \in \R^k, \ \sum_{i=1}^n a_i^{-1}|\scal{y}{v_i}| \leq 1\right\}.\]
Note that the set $\mathrm{diag}(a_1,\ldots,a_n)B_1^n \cap H$ is the image of $K$ under the linear injection $T: \R^k\to\R^n$ given by $Ty = [\scal{y}{v_i}]_{i=1}^n$, $y \in \R^k$, whose image is $H$. Therefore,
\begin{align*}
\vol_{H}\left( \mathrm{diag}(a_1,\ldots,a_n)B_1^n \cap H \right) &= \sqrt{\det(T^TT)}\vol_k(K) \\
&= \sqrt{\det\left(\sum_{i=1}^n v_iv_i^T\right)}\vol_k(K).
\end{align*}
Let us develop the formula for the volume of $K$. Plainly, $\|y\|_K = \sum_{i=1}^n a_i^{-1}|\scal{y}{v_i}|$, thus
\[
\vol_k(K) = \frac{1}{k!}\int_{\R^k} e^{-\|y\|_K} \dd y = \frac{1}{k!}\int_{\R^k} \prod_{i=1}^ne^{-a_i^{-1}|\scal{y}{v_i}|} \dd y.
\]
Using as in the proof of Lemma \ref{lm:volformula} that a standard symmetric exponential random variable with density $\frac{1}{2}e^{-|x|}$ has the same distribution as $\sqrt{2Y}G$, where $Y \sim \text{Exp}(1)$ and $G \sim N(0,1)$ are independent, we can write
\[
\frac{1}{2}e^{-|x|} = \E \frac{1}{\sqrt{2\pi}\sqrt{2Y}}e^{-\frac{x^2}{4Y}}.
\]
Taking i.i.d. copies $Y_1,\ldots,Y_n$ of $Y$, we obtain
\begin{align*}
\vol_k(K) &= \frac{1}{k!}\int_{\R^k} \left(\E_Y \prod_{i=1}^n \frac{1}{\sqrt{\pi}\sqrt{Y_i}}e^{-\frac{\scal{y}{v_i}^2}{4Y_ia_i^2}} \right) \dd y \\
&= \frac{2^{k/2}}{k!\cdot \sqrt{\pi}^{n-k}} \E_Y \left[\frac{1}{\sqrt{\prod_{i=1}^n Y_i}} \int_{\R^k} \frac{1}{\sqrt{2\pi}^k} e^{-\frac{1}{2}\scal{\left(\sum_{i=1}^n \frac{1}{2Y_ia_i^2}v_iv_i^T\right)y}{y}}\dd y \right] \\
&=  \frac{2^{k/2}}{k!\cdot \pi^{(n-k)/2}} \E_Y \left[\frac{1}{\sqrt{\prod_{i=1}^n Y_i}} \frac{1}{\sqrt{\det\left(\sum_{i=1}^n \frac{1}{2Y_ia_i^2}v_iv_i^T\right)}} \right].
\end{align*}
Plugging this back to the formula for the volume of the section $\mathrm{diag}(a_1,\ldots,a_n)B_1^n \cap H$ finishes the proof.
\end{proof}

Note that Lemma \ref{lm:volformula-2} uses $k$ dimensional vectors, whereas Lemma \ref{lm:volformula} uses $n-k$ dimensional vectors, where $k$ is the dimension of the section (subspace).


\begin{thebibliography}{9}

\bibitem{Bap} 
Bapat, R. B., 
Mixed discriminants of positive semidefinite matrices.
\emph{Linear Algebra Appl.} 126 (1989), 107--124.

\bibitem{Bar}
Barthe, F.,
Extremal Properties of Central Half-Spaces for Product Measures.
\emph{J.~Funct. Anal.} 182 (2001), 81--107

\bibitem{BLYZ}
B\"or\"oczky, K., Lutwak, E., Yang, D., Zhang, G.,
The log-Brunn-Minkowski inequality.
\emph{Adv. Math.} 231 (2012), 1974--1997.

\bibitem{CFM}
Cordero-Erausquin, D., Fradelizi, M., Maurey, B.,
The (B) conjecture for the Gaussian measure of dilates of symmetric convex sets and related problems. 
\emph{J. Funct. Anal.} 214 (2004), no. 2, 410--427.

\bibitem{CR}
Cordero-Erausquin, D., Rotem, L.,
personal communication (2017).

\bibitem{ENT}
Eskenazis, A., Nayar, P., Tkocz, T.,
Gaussian mixtures: entropy and geometric inequalities.
\emph{Ann. of Prob.} 46(5) 2018, 2908--2945.

\bibitem{MP}
Meyer, M., Pajor, A.,
Sections of the unit ball of $\ell_p^n$.
\emph{J. Funct. Anal.} 80 (1988) 109--123. 

\bibitem{Sar1}
Saroglou, Ch.,
Remarks on the conjectured log-Brunn-Minkowski inequality. 
\emph{Geom. Dedicata} 177 (2015), 353--365.

\bibitem{Sar2}
Saroglou, Ch., 
More on logarithmic sums of convex bodies. 
\emph{Mathematika} 62 (2016), no. 3, 818--841.


\end{thebibliography}
\end{document}